\documentclass[12pt]{amsart}
\usepackage{graphicx}
\usepackage{amsmath,amsthm, amsfonts,amssymb, stmaryrd,yfonts,pxfonts,pifont,eufrak, bbm}
\newtheorem{Cor}{Corollary}
 \newtheorem{Lemma}{Lemma}
 
 \newtheorem{ex}{Example}
 \newtheorem{Proposition}{Proposition}
 \theoremstyle{definition}
 
 \theoremstyle{remark}
 \newtheorem{Remark}[Lemma]{Remark}
 \numberwithin{equation}{subsection}

\begin{document}
\title[ON DYNAMICS GENERATED BY A UNIFORMLY CONVERGENT SEQUENCE OF MAPS]{ON DYNAMICS GENERATED BY A UNIFORMLY CONVERGENT SEQUENCE OF MAPS}%
\author{PUNEET SHARMA AND MANISH RAGHAV}
\address{Department of Mathematics, I.I.T. Jodhpur, Old Residency Road, Ratanada, Jodhpur-342011, INDIA}%
\email{puneet.iitd@yahoo.com, manishrghv@gmail.com }%




\begin{abstract}
In this paper, we study the dynamics of a non-autonomous dynamical system $(X,\mathbb{F})$ generated by a sequence $(f_n)$ of continuous self maps converging uniformly to $f$. We relate the dynamics of the non-autonomous system $(X,\mathbb{F})$ with the dynamics of $(X,f)$. We prove that if the family $\mathbb{F}$ commutes with $f$ and $(f_n)$ converges to $f$ at a "sufficiently fast rate", many of the dynamical properties for the systems $(X,\mathbb{F})$ and $(X,f)$ coincide. In the procees we establish equivalence of properties like equicontinuity, minimality and denseness of proximal pairs (cells) for the two systems. In addition, if $\mathbb{F}$ is feeble open, we establish equivalence of properties like transitivity, weak mixing and various forms of sensitivities. We prove that feeble openness of $\mathbb{F}$ is sufficient to establish equivalence of topological mixing for the two systems. We prove that if $\mathbb{F}$ is feeble open, dynamics of the non-autonomous system on a compact interval exhibits any form of mixing if and only if $(X,f)$ exhibits identical form of mixing. We also investigate dense periodicity for the two systems. We give examples to investigate  sufficiency/necessity of the conditions imposed. In the process we derive weaker conditions under which the established dynamical relation (between the two systems $(X,\mathbb{F})$ and $(X,f)$) is preserved.
\end{abstract}
\maketitle

\section{INTRODUCTION}

The theory of dynamical systems has been long used to study various physical or natural processes occurring in nature. Many of these processes have been modelled using discrete or continuous systems and their long term behavior has been investigated. The theory has found applications in a variety of fields like complex systems, control theory, biomechanics and cognitive sciences. While \cite{beer} used dynamical systems theory to study the agent environment interaction in the cognitive setting, in \cite{Hamill} authors used dynamical systems approach to lower extremity running injuries. In \cite{kohmoto} authors used the theory of discrete systems to model the chemical turbulence in a system. Although, the theory has been used extensively in a variety of fields, in most cases the physical/natural system is approximated using an autonomous system. Consequently, the governing rule $f$ for a dynamical system is assumed to be constant with time and the dynamics of the system $(X,f)$ is used to approximate the dynamics of the underlying system. Although such studies have resulted in good approximations of the underlying systems, better approximations can be obtained by allowing the governing rule to be time variant. As any general model approximating any natural/physical process is non-autonomous in nature, such a modification provides greater insight to the problem and hence results in a better approximation of the original system. Thus there is a strong need to investigate natural or physical systems using non-autonomous dynamical systems. As a result, some investigations for such a setting in the discrete case have been made and interesting results have been obtained. While \cite{sk1} study the topological entropy when the family $\mathbb{F}$ is equicontinuous or uniformly convergent, \cite{sk2} discusses minimality conditions for a non-autonomous system on a compact Hausdorff space while focussing on the case when the non-autonomous system is defined on a compact interval of the real line. In \cite{pm} authors investigate a non-autonomous system generated by a finite family of maps. In the process, they study properties like transitivity, weak mixing, topological mixing, existence of periodic points, various forms of sensitivities and Li-Yorke chaos. In \cite{jd} authors prove that if $f_n\rightarrow f$, in general there is no relation between chaotic behavior of the non-autonomous system generated by $f_n$ and the chaotic behavior of $f$. In \cite{of} authors investigate properties like weakly mixing, topological mixing, topological entropy and Li-Yorke chaos for the non-autonomous system.\\

In this paper, we investigate the dynamical behavior of a non-autonomous system generated by a sequence $\mathbb{F}=(f_n)$ of maps converging uniformly to $f$. In the process we establish relation between the dynamical behavior of the systems $(X,\mathbb{F})$ and $(X,f)$.  We prove that if $f$ commutes with each $f_n$ and $\sum \limits_{n=1}^{\infty} d_H(f_n,f) < \infty$, many of the dynamical properties coincide for the two systems. In the process, we establish that "commutative condition" and "fast convergence" ensures equivalence of properties like equicontinuity, minimality and denseness of proximal cells (pairs) for the two systems. Further, we prove that while feeble openness of $\mathbb{F}$ ensures equivalence of topological mixing for the two systems, additional assumptions of "commutative condition" and convergence at a "sufficiently fast rate" are needed to establish equivalence of properties like transitivity, weak mixing and various forms of sensitivities. We prove that if the limit map is an isometry, "commutative condition" is redundant and hence the established results hold good if $f_n$ converge at a sufficiently fast rate. We prove that if $\mathbb{F}$ is feeble open, dynamics of the non-autonomous system on a compact interval exhibits any form of mixing if and only if $(X,f)$ exhibits identical form of mixing. Further, we prove that any point periodic for $(X,\mathbb{F})$ is periodic for $(X,f)$ and hence a dense set of periodic points for $(X,\mathbb{F})$ ensures dense set of periodic points for $(X,f)$. We give example to show that existence of periodic points is not equivalent for the two systems. We also give examples to investigate the suffiency/necessity of conditions imposed for the results to hold good. In the process, we derive weaker conditions under which the derived relation between the two systems $(X,\mathbb{F})$ and $(X,f)$ is preserved. Before we move further, we give some of the basic concepts and definitions required.\\

Let $(X,d)$ be a compact metric space and let $\mathbb{F}=\{f_n: n \in \mathbb{N}\}$ be a family of continuous self maps on $X$.  For any initial seed $x_0$, any such family generates a non-autonomous dynamical system via the relation $x_{n}= f_n(x_{n-1})$. Throughout this paper, such a dynamical system will be denoted by $(X,\mathbb{F})$. For any $x\in X$, $\{ f_n \circ f_{n-1} \circ \ldots \circ f_1(x) : n\in\mathbb{N}\}$ defines the orbit of $x$. The objective of study of a non-autonomous dynamical system is to investigate the orbit of an arbitrary point $x$ in $X$. For notational convenience, let $\omega^n_{n+k}=f_{n+k}\circ f_{n+k-1}\circ\ldots\circ f_{n+1}$ and $\omega_n(x) = f_n\circ f_{n-1}\circ \ldots \circ f_1(x)$ (the state of the system after $n$ iterations). \\

A point $x$ is called \textit{periodic} for  $(X,\mathbb{F})$ if there exists $n\in\mathbb{N}$ such that $\omega_{nk}(x)=x$ for all $k\in \mathbb{N}$. The least such $n$ is known as the period of the point $x$. A system $(X,\mathbb{F})$ is called feeble open if for any non-empty open set $U$ in $X$, $int(f(U))\neq \phi$ for all $f\in \mathbb{F}$. The system $(X,\mathbb{F})$ is equicontinuous if for each $\epsilon>0$, there exists $\delta>0$ such that $d(x,y)<\delta$ implies $d(\omega_n(x),\omega_n(y))<\epsilon$ for all $n\in\mathbb{N},~~x,y\in X$. The system $(X,\mathbb{F})$ is \textit{transitive} (or $\mathbb{F}$ is transitive) if for each pair of non-empty open sets $U,V$ in $X$, there exists $n \in \mathbb{N}$ such that $\omega_n(U)\bigcap V\neq \phi$. The system $(X,\mathbb{F})$ is said to be \textit{minimal} if it does not contain any proper non-trivial subsystems. The system $(X,\mathbb{F})$ is said to be \textit{weakly mixing} if for any collection of non-empty open sets $U_1, U_2, V_1,V_2$ in $X$ there exists a natural number $n$ such that $\omega_n(U_i) \bigcap V_i \neq \phi$, $i=1,2$. Equivalently, we say that the system is weakly mixing if $\mathbb{F}\times\mathbb{F}$ is transitive. The system is said to be \textit{topologically mixing} if for every pair of non-empty open sets $U, V$ there exists a natural number $K$ such that $\omega_n(U) \bigcap V \neq \phi$ for all $n \geq K$. The system is said to be \textit{sensitive} if there exists a $\delta>0$ such that for each $x\in X$ and each neighborhood $U$ of $x$, there exists $n\in \mathbb{N}$ such that $diam(\omega_n(U))>\delta$. If there exists $K>0$ such that $diam(\omega_n(U))>\delta$ $~~\forall n\geq K$, then the system is \textit{cofinitely sensitive}. A pair $(x,y)$ is proximal for $(X,\mathbb{F})$ if $\liminf\limits_{n\rightarrow \infty} d(\omega_n(x),\omega_n(y))=0$. For any $x\in X$, the set $Prox_{\mathbb{F}}(x)=\{y: (x,y) \text{~~is proximal for~} (X,\mathbb{F})\}$ is called the proximal cell of $x$ in $(X,\mathbb{F})$. A system $(X,\mathbb{F})$ is said to exhibit dense set of proximal pairs if the set of pairs proximal for $(X,\mathbb{F})$ is dense in $X\times X$. A set $S$ is said to be $\delta$-\textit{scrambled} in $(X,\mathbb{F})$ if for any distinct $x,y\in S$, $\limsup\limits_{n\rightarrow \infty} d(\omega_n(x),\omega_n(y))>\delta$ but $\liminf\limits_{n\rightarrow \infty} d(\omega_n(x),\omega_n(y))=0$. A system $(X,\mathbb{F})$ is Li-Yorke sensitive if there exists $\delta>0$ such that for each $x\in X$ and each neighborhood $U$ of $x$ there exists $y\in U$ such that $(x,y)$ is a $\delta$-scrambled set. For any $x\in X$, let $LY_{\mathbb{F}}(x)=\{y\in X : (x,y) \text{~~is a Li-Yorke pair for~} (X,\mathbb{F})\}$ is called the Li-Yorke cell of $x$. It may be noted that in case the $f_n$'s coincide, the above definitions coincide with the known notions of an autonomous dynamical system. See \cite{bc,bs,de} for details.\\

Let $X$ be a compact space and let $\mathcal{K}(X)$ denote the collection of all non-empty compact subsets of $X$. For any $A,B\in\mathcal{K}(X)$ define $D_H(A,B)=\inf\{\epsilon>0:A\subset S(B,\epsilon) \text{~~and~~} B\subset S(A,\epsilon)\}$ where $S(A,\epsilon)=\bigcup\limits_{x\in A} S(x,\epsilon)$ is the $\epsilon$-ball around $A$. Then $D_H$ defines a metric on $\mathcal{K}(X)$ and is known as the Hausdorff metric. It is known that a system $(X,f)$ is a weakly mixing (topological mixing) if and only if for any compact set $K$ with non-empty interior $\limsup\limits_{n\rightarrow\infty} f^n(K)=X$ ($\lim\limits_{n\rightarrow\infty} f^n(K)=X$) with respect to the metric $D_H$.\\

Let $(X,d)$ be a compact metric space and let $C(X)$ denote the collection
of continuous self maps on $X$. For any $f,g\in C(X)$, define,

\centerline{$ D (f,g) = \sup \limits_{x\in X} d(f(x),g(x))$}

It is easily seen that $D$ defined above is a metric on $C(X)$ and is known as the \textit{Supremum metric}. It can be seen that a sequence $(f_n)$ in $C(X)$ converges to $f$ in $(C(X),D)$ if and only if $f_n$ converges to $f$ uniformly on $X$ and hence the topology generated by the metric defined above is known as the topology of uniform convergence. A collection of sequences $\{(f_j^i) : i\in \mathfrak{I}\}$ converges collectively to $\{g^i: i\in\mathfrak{I}\}$ with respect to the metric $D$ if for each $\epsilon>0$, there exists $n_0\in\mathbb{N}$ such that $D(f_j^i,g^i)<\epsilon~~\forall j\geq n_0,~~i\in\mathfrak{I}$.

\section{Main Results}

Throughout this section, the family $\mathbb{F}$ and the limit map $f$ are assumed to be surjective.

\begin{Proposition}\label{ap1}
Let $(X,\mathbb{F})$ be a non-autonomous system generated by $\mathbb{F}$ and let $f$ be any continuous self map on $X$. If the family $\mathbb{F}$ commutes with $f$ then for any $x\in X$ and any $k\in \mathbb{N}$, $d(\omega_{k}(x),f^{k}(x))\leq \sum \limits_{i=1}^k D(f_{i},f)$.
\end{Proposition}

\begin{proof}
Let $x\in X$ and $n$ be a natural number. As $f_k$'s commute with $f$, $d(f_2\circ f_1(x),f^2(x))\leq d(f_{2}\circ f_{1}(x),f\circ f_{1}(x))+ d(f \circ f_{1}(x),f\circ f(x))= d(f_{2}\circ f_{1}(x),f\circ f_{1}(x))+ d(f_1 \circ f(x),f\circ f(x))\leq D(f_{2},f)+D(f_{1},f)$. Proceeding inductively, if $d(f_{n}\circ f_{n-1}\circ\ldots \circ f_{1}(x),f^{n}(x))\leq \sum \limits_{i=1}^n D(f_{i},f)$, then, $d(f_{n+1}\circ f_n\circ\ldots\circ f_{1}(x),f^{n+1}(x))\leq d(f_{n+1}\circ f_n\circ\ldots\circ f_{1}(x),f\circ f_n\circ\ldots\circ f_{1}(x))+ d(f\circ f_n\circ\ldots\circ f_{1}(x),f^{n+1}(x))= d(f_{n+1}\circ f_n\circ\ldots\circ f_{1}(x),f\circ f_n\circ\ldots\circ f_{1}(x)+ d(f_{n}\circ f_{n-1}\circ \ldots\circ f_{1}\circ f(x),f^n\circ f(x)) \leq D(f_{n+1},f)+ \sum \limits_{i=1}^n D(f_{i},f)$ (by induction). Hence for any $k\in \mathbb{N}$, $d(f_{k}\circ f_{k-1}\circ\ldots \circ f_{1}(x),f^{k}(x))\leq \sum \limits_{i=1}^k D(f_{i},f)$ or $d(\omega_{k}(x),f^{k}(x))\leq \sum \limits_{i=1}^k D(f_{i},f)$.
\end{proof}

\begin{Remark}
The above result captures the deviations in the orbit of the non-autonomous system, when approximated by an autonomous system $(X,f)$. Consequently, the result provides an upper bound for the error in the approximation of the trajectory of the non-autonomous system, when approximated using the system $(X,f)$. It may be noted that as error after $k$ iterations is bounded by $\sum \limits_{i=1}^k D(f_{i},f)$, if $(f_n)$ converges to $f$ with order more than $O(\frac{1}{n^r})~~(r>1)$, the total error for the system remains bounded. Further, if $(f_n)$ converges faster than $O(\frac{1}{n^r})$, approximating the trajectory after its observation for finite time leads to a better approximation. In particular, an arbitrarily small bound for the error can be obtained if approximation process is initiated after finitely many iterations. Thus, we get the following corollary.
\end{Remark}

\begin{Cor}\label{apcor}
Let $(X,\mathbb{F})$ be a non-autonomous system generated by $\mathbb{F}$ and let $f$ be any continuous self map on $X$. If the family $\mathbb{F}$ commutes with $f$ then for any $x\in X$ and any $k,n\in \mathbb{N}$, $d(\omega_{n+k}(x),f^{k}(\omega_n(x)))\leq \sum \limits_{i=1}^k D(f_{n+i},f)$
\end{Cor}

\begin{proof}
The result is trivial by approximating the trajectory of $x_0=\omega_n(x)$ for $k$ iterations when $\mathbb{F}= \{f_{n+1}: n\in \mathbb{N} \}$.
\end{proof}

\begin{Remark}
The corollary above establishes that if the family $\mathbb{F}$ commutes with $f$ then for any $k,n\in\mathbb{N}$, $D(\omega_{n+k},f^k(\omega_n))\leq \sum \limits_{i=1}^k D(f_{n+i},f)$ or in other words, $D(\omega^n_{n+k}(\omega_n),f^k(\omega_n))\leq \sum \limits_{i=1}^k D(f_{n+i},f)~~\forall (k,n)\in\mathbb{N}^2$. Consequently, if $f_n$ converges to $f$ at a sufficiently fast rate and each $f_i$ is surjective then $\omega^n_{n+k}$ converges to $f^k$ collectively (with respect to the metric $D$). It may be noted that if $f_n$ converges uniformly to $f$ then for each $k\in\mathbb{N}$, $\omega^n_{n+k}$ converges uniformly to $f^k$. However, if the family $\mathbb{F}$ commutes with $f$ the convergence is collective and hence a much stronger form of convergence can be concluded which need not hold good in general . Hence we get the following corollary.
\end{Remark}

\begin{Cor}\label{apcor2}
Let $(X,\mathbb{F})$ be a non-autonomous system generated by a family $\mathbb{F}$ and let $f$ be any continuous self map on $X$. If the family $\mathbb{F}$ commutes with $f$ and $\sum \limits_{n=1}^{\infty} D(f_n,f) < \infty$ then $\{(\omega^n_{n+k}): k\in\mathbb{N}\}$ converges to $\{f^k:k\in\mathbb{N}\}$ collectively (with respect to the metric $D$).
\end{Cor}

\begin{proof}
If the family $\mathbb{F}$ commutes with $f$, by corollary \ref{apcor} we have $D(\omega^n_{n+k}(\omega_n),f^k(\omega_n))\leq \sum \limits_{i=1}^k D(f_{n+i},f)$. Consequently, if the family $\mathbb{F}$ is surjective then each $\omega_n$ is surjective and hence  $D(\omega^n_{n+k},f^k)\leq \sum \limits_{i=1}^k D(f_{n+i},f)$. Further, if $\sum \limits_{n=1}^{\infty} D(f_n,f) < \infty$, then for each $\epsilon>0$ there exists $r\in\mathbb{N}$ such that $\sum \limits_{i=1}^{\infty} D(f_{r+i},f)<\epsilon$ which implies $D(\omega^n_{n+k},f^k)<\epsilon~~ \forall k\in\mathbb{N}, n\geq r$ and hence the convergence is collective.
\end{proof}

\begin{Remark}
The corollary establishes that if the family $\mathbb{F}$ commutes with $f$ and $\sum\limits_{n=1}^{\infty} D(f_n,f)<\infty$ then $\{(\omega^n_{n+k}):k\in\mathbb{N}\}$ converges collectively to $\{f^k:k\in\mathbb{N}\}$ with respect to the metric $D$. However the proof uses the fact that the stated conditions ensure that for each $\epsilon>0$ there exists $n_0\in\mathbb{N}$ such that $d(\omega_{n+k}(x),f^{k}(\omega_n(x)))<\epsilon$ for any $k\in \mathbb{N}, n\geq n_0$ and $x\in X$,  which in turn ensures the collective convergence. Conversely, if $\{\omega^n_{n+k}:k\in\mathbb{N}\}$ converges collectively then for each $\epsilon>0$ there exists $n_0\in\mathbb{N}$ such that $D(\omega^n_{n+k},f^k)<\epsilon$ for all $k\in\mathbb{N}$ and $n\geq n_0$. Further, if each $f_i$ is surjective then $\omega_n$ are surjective and hence $d(\omega^n_{n+k}(\omega_n(x)),f^k(\omega_n(x)))<\epsilon$ for all $x\in X,~~k\in\mathbb{N}$ and $n\geq n_0$ or $d(\omega_{n+k}(x),f^n(\omega_n(x)))<\epsilon$ for all $x\in X,~~k\in\mathbb{N}$ and $n\geq n_0$ and hence two statements are equivalent.
\end{Remark}

\begin{Proposition}\label{eq}
Let $(X,\mathbb{F})$ be a non-autonomous system generated by a family $\mathbb{F}$ commuting with $f$. If $\sum \limits_{n=1}^{\infty} D(f_n,f) < \infty$ then $(X,f)$ is equicontinuous $\Rightarrow$ $(X,\mathbb{F})$ is equicontinuous. Further, if $f_i's$ are bijective then, $(X,\mathbb{F})$ is equicontinuous $\Rightarrow (X,f)$ is equicontinuous.
\end{Proposition}

\begin{proof}
Let $(X,f)$ be equicontinuous and let $\epsilon>0$ be given. By equicontinuity, there exists $\delta>0$ such that $d(x,y)<\delta$ implies $d(f^n(x),f^n(y))<\frac{\epsilon}{3}$ for all $n\in\mathbb{N}$.  As $\sum \limits_{n=1}^{\infty} D(f_n,f) < \infty$, choose $r\in\mathbb{N}$ such that $\sum \limits_{n=r}^{\infty} D(f_n,f) <\frac{\epsilon}{3}$ and hence for any $x\in X$, by corollary \ref{apcor}~~ we have, $d(\omega_{r+k}(x),f^k(\omega_r(x)))<\sum \limits_{i=1}^k D(f_{r+i},f)<\frac{\epsilon}{3}$. As $\omega_r$ is continuous, choose $\eta_r>0$ such that $d(x,y)<\eta_r$ implies $d(\omega_r(x),\omega_r(y))<\delta$ and hence $d(f^k(\omega_r(x)),f^k(\omega_r(y)))<\frac{\epsilon}{3}$ for all $k\in\mathbb{N}$. By triangle inequality, $d(x,y)<\eta_r$ implies $d(\omega_{r+k}(x), \omega_{r+k}(y))<\epsilon$ for all $k\in \mathbb{N}$ or $d(\omega_{k}(x),\omega_{k}(y))<\epsilon$ for all $k\geq r+1$. Further, as $\{\omega_1,\omega_2,\ldots,\omega_r\}$ is a finite set, there exists $\eta_r'>0$ such that $d(x,y)<\eta_r'$ forces $d(\omega_i(x),\omega_i(y))<\epsilon$ for $i=1,2,\ldots,r$ and hence choosing $\eta=\min \{\eta_r,\eta_r'\}$ ensures equicontinuity of  $(X,\mathbb{F})$.


Conversely, let $f_i$'s be bijective and let $\epsilon>0$ be given. By equicontinuity of $(X,\mathbb{F})$, there exists $\delta>0$ such that $d(x,y)<\delta$ implies $d(\omega_n(x),\omega_n(y))<\frac{\epsilon}{3}$ for all $n\in\mathbb{N}$. As $\sum \limits_{n=1}^{\infty} D(f_n,f) < \infty$, choose $r\in\mathbb{N}$ such that $\sum \limits_{n=r}^{\infty} D(f_n,f) <\frac{\epsilon}{3}$. Also bijectivity of $f_i$ implies $\omega_r$ is a homeomorphism and thus there exists $\eta_r>0$ such that $d(x,y)<\eta_r$ implies $d(\omega_r^{-1}(x), \omega_r^{-1}(y))<\delta$. Let $x,y\in X$ such that $d(x,y)<\eta_r$. Let $x_r=\omega_r^{-1}(x)$ and $y_r=\omega_r^{-1}(y)$. It may be noted that $d(f^k(x),f^k(y))=d(f^k(\omega_r(x_r)),f^k(\omega_r(y_r)))\leq d(f^k(\omega_r(x_r)), \omega_{r+k}(x_r))+ d(\omega_{r+k}(x_r), \omega_{r+k}(y_r))+d(\omega_{r+k}(y_r),f^k(\omega_r(y_r)))$. By corollary \ref{apcor},~~ $d(\omega_{r+k}(x_r), f^k(\omega_r(x_r)))\leq \sum \limits_{i=1}^k D(f_{r+i},f)<\frac{\epsilon}{3}$ and $d(\omega_{r+k}(y_r), f^k(\omega_r(y_r)))\leq\sum\limits_{i=1}^k D(f_{r+i},f)<\frac{\epsilon}{3}$. Also $d(x,y)<\eta_r$ implies $d(x_r,y_r)<\delta$ and hence by equicontinuity we have $d(\omega_{r+k}(x_r),\omega_{r+k}(y_r))<\frac{\epsilon}{3}$. By triangle inequality $d(f^k(\omega_r(x_r)),f^k(\omega_r(y_r)))<\epsilon$ or $d(f^k(x),f^k(y))<\epsilon$. As the proof holds for any $k\in\mathbb{N}$, $(X,f)$ is equicontinuous.
\end{proof}

\begin{Remark}
The above result establishes the relation between equicontinuity of the two systems $(X,f)$ and $(X,\mathbb{F})$. While a "sufficiently fast rate of convergence" is enough to establish the equicontinuity of $(X,\mathbb{F})$ from equicontinuity of $(X,f)$, the proof uses additional assumption of bijectivity of the generating functions $f_i$ to establish the converse. However, it may be noted that the converse part does not use the bijectivity completely and the result holds good under a much weaker assumption. In particular, if every pair of nearby points can be obtained from nearby points, a similar proof of the converse holds good and the equicontinuity of the two systems is still equivalent. We say that a system $(X,\mathbb{F})$ exhibits "nearness criteria" if for each $\epsilon>0$, there exists $\delta >0$, such that for every pair of points $x,y\in X$ with $d(x,y)<\delta$, there exists $u,v\in X$ and $r\in\mathbb{N}$ such that $\omega_r(u)=x,\omega_r(v)=y$ and $d(u,v)<\epsilon$. The definition is analogous to the continuity of the inverse and is equivalent to continuity of each fibre of the inverse. In view of this definition, the above result yields the following corollary.
\end{Remark}

\begin{Cor}\label{eq2}
Let $(X,\mathbb{F})$ be a non-autonomous system generated by a family $\mathbb{F}$ commuting with $f$. If $(X,\mathbb{F})$ satisfies nearness criteria and $\sum \limits_{n=1}^{\infty} D(f_n,f) < \infty$ then $(X,f)$ is equicontinuous $\Leftrightarrow$ $(X,\mathbb{F})$ is equicontinuous.
\end{Cor}

\begin{Proposition}\label{min}
Let $(X,\mathbb{F})$ be a non-autonomous system generated by a sequence $\mathbb{F}$ commuting with $f$. If $\sum \limits_{n=1}^{\infty} D(f_n,f) < \infty$ then $(X,f)$ is minimal $\Leftrightarrow$ $(X,\mathbb{F})$ is minimal.
\end{Proposition}

\begin{proof}
Let $(X,f)$ be minimal and let $x\in X$. Let $U$ be any non-empty open set in $X$. Choose $u\in U$ and $\epsilon>0$ such that $S(u,\epsilon)\subset U$. As $\sum \limits_{n=1}^{\infty} D(f_n,f) < \infty$, there exists $r\in\mathbb{N}$ such that $\sum \limits_{n=r}^{\infty} D(f_n,f) < \frac{\epsilon}{2}$. As $(X,f)$ is minimal, orbit of $\omega_r(x)$(under $f$) is dense in $X$ and there exists $k\in \mathbb{N}$ such that $f^k(\omega_r(x))\in S(u,\frac{\epsilon}{2})$. Also by corollary \ref{apcor}, $d(\omega_{r+k}(x),f^k(\omega_r(x)))\leq \sum \limits_{i=1}^{k} D(f_{r+i},f)<\frac{\epsilon}{2}$. Thus, by triangle inequality, $d(\omega_{r+k}(x),u)\leq d(\omega_{r+k}(x),f^k(\omega_r(x)))+d(f^k(\omega_r(x)),u)<\epsilon$ and hence $\omega_{r+k}(x)\in S(u,\epsilon)\subset U$. As the proof holds good for any non-empty open set $U$ of $X$, orbit of $x$ (under $\mathbb{F}$) is dense in $X$. As the proof holds good for any $x\in X$ we establish that $(X,\mathbb{F})$ is minimal.\\

Conversely, let $(X,\mathbb{F})$ be minimal and let $x\in X$. Let $U$ be any non-empty open set in $X$. Choose $u\in U$ and $\epsilon>0$ such that $S(u,\epsilon)\subset U$. As $\sum \limits_{n=1}^{\infty} D(f_n,f) < \infty$, there exists $r\in\mathbb{N}$ such that $\sum \limits_{n=r}^{\infty} D(f_n,f) < \frac{\epsilon}{2}$. As $(X,\mathbb{F})$ is minimal, orbit of any $y\in\omega_r^{-1}(x)$(under $\mathbb{F}$) is dense in $X$ and hence intersects $S(u,\frac{\epsilon}{2})$. Further, as the orbit intersects $S(u,\eta)$ for each $\eta$, the set $\{n: \omega_n(y)\in S(u,\frac{\epsilon}{2})\}$ is infinite and there exists $k\in \mathbb{N}$ such that $\omega_{r+k}(y)\in S(u,\frac{\epsilon}{2})$. Also by corollary \ref{apcor}, $d(\omega_{r+k}(y),f^k(\omega_r(y)))\leq\sum\limits_{i=1}^k D(f_{r+i},f)<\frac{\epsilon}{2}$ and hence by triangle inequality, $d(f^k(\omega_r(y)),u)<\epsilon$ or $d(f^k(x),u)<\epsilon$ and hence $f^k(x)\in U$. As the proof holds good for any non-empty open set $U$, orbit of $x$(under $f$) is dense in $X$. As the proof holds good for any $x\in X$, $(X,f)$ is minimal.
\end{proof}

\begin{Proposition}\label{tt}
Let $(X,\mathbb{F})$ be a non-autonomous system generated by a family $\mathbb{F}$ of feeble open maps commuting with $f$. If $\sum \limits_{n=1}^{\infty} D(f_n,f) < \infty$ then $(X,f)$ is transitive $\Leftrightarrow$ $(X,\mathbb{F})$ is transitive.
\end{Proposition}
\begin{proof}
 Let $\epsilon>0$ be given and let  $U=S(x,\epsilon)$ and $V=S(y,\epsilon)$ be two non-empty open sets in $X$. As $\sum \limits_{n=1}^{\infty} D(f_n,f) < \infty$, $\exists$ $r\in \mathbb{N}$ such that $\sum \limits_{n=r}^{\infty} D(f_n,f) < \frac{\epsilon}{2}$. As the family $\mathbb{F}$ is feeble open, $int(\omega_r(U))$ is non-empty open and thus by transitivity of $(X,f)$, for open sets $U'=int(\omega_r(U))$ and $V'=S(y,\frac{\epsilon}{2})$ there exists $m\in \mathbb{N}$ such that $f^m(U')\cap V' \neq\phi$. Consequently there exists $u'\in U'$ such that $f^m(u')\in V'$. Further, as $U'=int(\omega_r(U))$, there exists $u\in U$ such that $u'=\omega_r(u)$ and hence $f^m (\omega_r(u))\in V'$.

Also, from corollary \ref{apcor}, we have $d(\omega_{m+r}(u), f^m(\omega_r(u)))\leq \sum \limits_{i=1}^m D(f_{r+i},f)<\frac{\epsilon}{2}$. By triangle inequality $d (y,\omega_{m+r}(u))<\epsilon$ and hence $\omega_{m+r}(U)\cap V\neq\phi$. As the proof holds for any pair of non-empty open sets $S(x,\epsilon), S(y,\epsilon)$ in $X$, the proof holds for any pair of non-empty open sets in $X$. Hence $(X,\mathbb{F})$ is transitive.\\

Conversely, let $\epsilon>0$ be given and let  $S(x,\epsilon)$ and $S(y,\epsilon)$ be two non-empty open sets in $X$. As $\sum \limits_{n=1}^{\infty} D(f_n,f) < \infty$, choose $r\in \mathbb{N}$ such that $\sum \limits_{n=r}^{\infty} D(f_n,f) < \frac{\epsilon}{2}$. Further, transitivity of a system forces any open set $U$ to visit $S(y,\frac{1}{m})$ for each $m$ and hence the set of times when any non-empty open set $U$ visits $S(y,\epsilon)$ is infinite. Applying transitivity to open sets $U=\omega_r^{-1}(S(x,\epsilon))$ and $V=S(y,\frac{\epsilon}{2})$ we can choose $k$ such that $\omega_{r+k}(U)\cap V\neq\phi$. Consequently, there exists $u\in U$ such that $d(\omega_{r+k}(u),y)<\frac{\epsilon}{2}$. Also by corollary \ref{apcor}, $d(\omega_{r+k}(u),f^k(\omega_r(u)))< \sum \limits_{i=1}^k D(f_{r+i},f)<\frac{\epsilon}{2}$ and hence by triangle inequality $d(y,f^k(\omega_r(u)))<\epsilon$. As $\omega_r(u)\in S(x,\epsilon)$ we have $f^k(S(x,\epsilon))\cap S(y,\epsilon)\neq\phi$. As the proof holds for any choice of non-empty open sets $S(x,\epsilon)$ and $S(y,\epsilon)$, the system $(X,f)$ is transitive.
\end{proof}

\begin{Remark}
The above proofs establish the equivalence of properties like equicontinuity (minimality and transitivity) for the two systems under the stated conditions. However, the proof uses the identity $D(\omega_{n+k},f^k(\omega_n))<\sum\limits_{i=1}^{\infty} D(f_{n+i},f)~~\forall k,n\in\mathbb{N}$ to establish the results. Consequently, the proof utilizes the collective convergence of the sequences $\{(\omega^n_{n+k}): k\in\mathbb{N}\}$ and does not make use of the commutativity condition or the rate of convergence of the sequence $(f_n)$ explicitly. Hence the established results hold good under a weaker assumption of collective convergence. Thus, the conditions imposed are sufficient in nature and similar conclusions can hold good even in absence of these conditions. We now give some examples in support of our claim.
\end{Remark}

\begin{ex}\label{eqex}
Let $X=S^1$ be the unit circle and $\alpha\in\mathbb{R}$ be an irrational. For each $n\in\mathbb{N}$, let $f_n:S^1\rightarrow S^1$ be defined as \\

$f_n(\theta) = \left\{%
\begin{array}{ll}
            \theta+\alpha+\frac{2}{n+1}  & 0\leq \theta\leq 1 \text{~~if $n$ is odd} \\
            \theta+\alpha-\frac{2}{n} & 0\leq\theta\leq 1\text{~~if $n$ is even}
\end{array} \right.$

Then $f_n$ converges uniformly to an irrational rotation and $\sum \limits_{n=1}^{\infty} D(f_n,f)$ is infinite. However, as $\omega_{2n}$ rotates the point $theta$ by $2n\alpha$, both $(X,\mathbb{F})$ and $(X,f)$ are minimal although the sequence does not converge at sufficiently fast rate. As the system $(X,\mathbb{F})$ is equicontinuous (transitive), the example shows that rate of convergence is not a necessary condition for properties like equicontinuity (minimality or transitivity) to be equivalent for the two systems.
\end{ex}

\begin{ex}
Let $X=\{0,1\}^{\mathbb{N}}$ be the collection of all one-sided sequences of $0$ and $1$. For any $x,y\in X$, $d(x,y)=\frac{1}{k}$ (where $k$ is the least positive integer satisfying $x_k\neq y_k$) defines a metric on $X$ and generates the product topology. For each $n\in\mathbb{N}$ define $\phi_n:X\rightarrow X$ be defined as $\phi_n(x)=y_n$ (where $y_n$ is obtained by deleting $n$-th entry from the sequence $x$) and let $f:X\rightarrow X$ be defined as $f(x)=x+(100\ldots)$. Let $\mathbb{F}=\{f_{n}=f\circ\phi_{n}:n\in\mathbb{N}\}$. Then, $f_n$ converges uniformly to $f$ and $\sum\limits_{n=1}^{\infty} D(f_n,f)$ is infinite. Further, the family $\mathbb{F}$ does not commute with $f$ and hence the system does not satisfy any of the imposed conditions. However, both $(X,\mathbb{F})$ and $(X,f)$ are minimal (equicontinuous) and hence the properties are preserved in the absence of imposed conditions.
\end{ex}


\begin{Remark}
The above examples show that equicontinuity (minimality and transitivity) of the two systems can be equivalent without the imposed conditions. However, collective convergence plays a vital role in establishing the equivalence and is needed for the proofs to hold good. It may be noted that as $D(\omega^n_{n+r+1},f^{r+1})= D(f_{n+r+1}(\omega^n_{n+r}),f(f^{r})) \leq D(f_{n+r+1}(\omega^n_{n+r}),f(\omega^n_{n+r}))+D(f(\omega^n_{n+r}), f(f^r))$, if $f$ is an isometry then under "fast convergence" of $(f_n)$,  the notion of uniform convergence implies (and hence is equivalent to) the notion of collective convergence. Consequently, if $f$ is an isometry, the proofs establishing equivalence of equicontinuity (minimality and transitivity) hold good in the absence of "commutative condition" and hence the dynamical behavior of the two systems is equivalent in this sense. Thus we get the following results.
\end{Remark}

\begin{Proposition}\label{iso}
Let $(X,\mathbb{F})$ be a non-autonomous system generated by a family $\mathbb{F}$ and let $f$ be any continuous self map on $X$. If $f$ is an isometry and $\sum \limits_{n=1}^{\infty} D(f_n,f) < \infty$ then $\{(\omega^n_{n+k}): k\in\mathbb{N}\}$ converges to $\{f^k:k\in\mathbb{N}\}$ collectively (with respect to the metric $D$).
\end{Proposition}

\begin{proof}
We first prove that if $f$ is an isometry then $D(\omega^n_{n+k},f^k)\leq \sum \limits_{i=1}^k D(f_{n+i},f)$. As $\omega^n_{n+1}=f_{n+1}$, $D(\omega^n_{n+1},f)=D(f_{n+1},f)$ and hence the result is true for $k=1$. It may be noted that $D(\omega^n_{n+r+1},f^{r+1})= D(f_{n+r+1}(\omega^n_{n+r}),f(f^{r}))\leq D(f_{n+r+1}(\omega^n_{n+r}),f(\omega^n_{n+r}))+D(f(\omega^n_{n+r}), f(f^r))\leq D(f_{n+r+1},f)+D(\omega^n_{n+r},f^r)$ (as $f$ is an isometry). Thus if the claim holds for $k=r$ then, $D(\omega^n_{n+r+1},f^{r+1})\leq  \sum \limits_{i=1}^{r+1} D(f_{n+i},f)$ and hence $D(\omega^n_{n+k},f^k)\leq \sum \limits_{i=1}^k D(f_{n+i},f)$ for any $k\in\mathbb{N}$. As $\sum \limits_{n=1}^{\infty} D(f_n,f) < \infty$, for any $\epsilon>0$, there exists $n_0\in\mathbb{N}$ such that $\sum\limits_{i=1}^{\infty} D(f_{n_0+i},f)<\epsilon$. Consequently, $D(\omega^n_{n+k},f^k)<\epsilon~~\forall k\in\mathbb{N}, n\geq n_0$ and hence the family $\{(\omega^n_{n+k}): k\in\mathbb{N}\}$ converges collectively to $\{f^k:k\in\mathbb{N}\}$ (with respect to the metric $D$).
\end{proof}

\begin{Cor}
Let $(X,\mathbb{F})$ be a non-autonomous system generated by a family $\mathbb{F}$. If $f$ is an isometry and $\sum \limits_{n=1}^{\infty} D(f_n,f) < \infty$ then $(X,f)$ exhibits equicontinuity (minimality) $\Leftrightarrow$ $(X,\mathbb{F})$ exhibits the same. Further, if $\mathbb{F}$ is feeble open then $(X,f)$ is transitive $\Leftrightarrow$ $(X,\mathbb{F})$ is transitive.
\end{Cor}

\begin{Remark}
The above result establishes that if $f$ is an isometry and $f_n$'s converge to $f$ at a "sufficiently fast rate" then uniform convergence is equivalent to "collective convergence", which in general is a stronger notion of convergence. It may be noted that the above proof holds good when $f$ does not expand any region in the space i.e.,$d(f(x),f(y)\leq d(x,y)~~\forall x,y\in X$ and hence equivalence of uniform convergence with collective convergence is established for a larger class of maps. Consequently, if $f$ is "shrinking" (does not expand any region of the space) and $f_n$'s converge to $f$ at a sufficiently fast rate then properties like equicontinuity, minimality and transitivity are equivalent for the two sysyems. Hence we get the following corollary.
\end{Remark}

\begin{Cor}
Let $(X,\mathbb{F})$ be a non-autonomous system generated by a family $\mathbb{F}$. If $f$ is shrinking and $\sum \limits_{n=1}^{\infty} D(f_n,f) < \infty$ then $(X,f)$ exhibits equicontinuity (minimality) $\Leftrightarrow$ $(X,\mathbb{F})$ exhibits the same. Further, if $\mathbb{F}$ is feeble open then $(X,f)$ is transitive $\Leftrightarrow$ $(X,\mathbb{F})$ is transitive.
\end{Cor}

\begin{Remark}
The above proof establishes equivalence of properties like equicontinuity (minimality and transitivity) in the absence of commutativity condition. However, the proof requires the limit map $f$ to be an isometry (shrinking) which ensures collective convergence and hence establishes the equivalence of mentioned dynamical properties for the two systems. It may be noted that to establish the equivalence of transitivity for the two systems, the proof requires the family $\mathbb{F}$ to be feeble open (along with collective convergence of $(\omega^n_{n+k})$). However, the converse part does not use the feeble openness of the maps $f_n$ and hence transitivity of the non-autonomous system is carried forward to $(X,f)$ even when the family $\mathbb{F}$ is not feeble open. Further, if $N_f(U,V)/N_{\mathbb{F}}(U,V)$ denotes the set of times when an open set $U$ visits $V$ under $f$/$\mathbb{F}$, the proof establishes that for each pair of open sets $U,V$, there exists a pair of open sets $U',V'$ of open sets such that the set of times of interactions of $U$ and $V$ under $\mathbb{F}$ contains the translates (by a fixed number depending on the diameter of open sets and the rate of convergence) of set of times of interactions of $U'$ and $V'$ under $f$(and vice-versa). As the argument depends on the diameter of open sets and not on open sets themselves, a argument similar to above holds good for two pairs $U_1,U_2$ and $V_1,V_2$ of non-empty open sets in $X$ and hence weakly mixing is also equivalent for the two systems. We include the proof for the sake of completion.
\end{Remark}

\begin{Proposition}\label{wm}
Let $(X,\mathbb{F})$ be a non-autonomous system generated by a family $\mathbb{F}$ of feeble open maps commuting with $f$. If $\sum \limits_{n=1}^{\infty} D(f_n,f) < \infty$ then $(X,f)$ is weakly mixing $\Leftrightarrow$ $(X,\mathbb{F})$ is weakly mixing.
\end{Proposition}

\begin{proof}
Let $U_1,U_2$ and $V_1,V_2$ be a pair of two non-empty open sets in $X$. Without loss of generality, let $U_i=S(u_i,\epsilon)$ $V_i=S(v_i,\epsilon)$ for $i=1,2$. As $\sum \limits_{n=1}^{\infty} D(f_n,f) < \infty$, $\exists$ $r\in \mathbb{N}$ such that $\sum \limits_{n=r}^{\infty} D(f_n,f) < \frac{\epsilon}{2}$. As $f_i$ are feeble open $U_i'=$int$(\omega_r(U_i))$ is open, applying weak mixing of $f$ to pairs $U_1',U_2'$ and $V_1',V_2'$ (where $V_i'=S(v_i,\frac{\epsilon}{2})$), there exists $k\in\mathbb{N}$ such that $f^k(U_i')\cap V_i'\neq \phi$. A proof similar to the transitive case establishes existence of $u_i\in U_i$ such that $\omega_{k+r}(u_i)\in V_i$ for $i=1,2$ and hence the system $(X,\mathbb{F})$ is weakly mixing.

Conversely, let $\epsilon>0$ be given and let  $U_1,U_2$ and $V_1,V_2$ be two non-empty open sets in $X$. Once again, as $\sum \limits_{n=1}^{\infty} D(f_n,f) < \infty$, $\exists$ $r\in \mathbb{N}$ such that $\sum \limits_{n=r}^{\infty} D(f_n,f) < \frac{\epsilon}{2}$. Further, continuity of $\omega_r$ implies $U_i'=\omega_r^{-1}(U_i)$ is open in $X$. Applying weak mixing of the system $(X,\mathbb{F})$ to pairs of open sets $U_1',U_2'$ and $V_1',V_2'$ (where $V_i'=S(v_i,\frac{\epsilon}{2})$), a proof similar to the transitive case establishes existence of $u_i\in U_i$ and $k\in\mathbb{N}$ such that $f^k(u_i)\in V_i$ and hence $(X,f)$ is weakly mixing.
\end{proof}

\begin{Remark}\label{r}
The above result generalises proposition $4$ and establishes the equivalence of weakly mixing for the two systems under stated conditions. As a result, feeble openness of $f_n$'s is redundant for the converse part and hence the proof holds good even when the maps $f_n$ generating the non-autonomous system are not feeble open. It may be noted that if $f$ is topologically mixing map then for any non-empty open set $U$, $\lim \limits_{n\rightarrow\infty} f^n(U)=X$. Consequently, there exists $\epsilon>0$ such that if $U$ is open set such that $D_H(U,X)<\epsilon$ then $D_H(f(U),X)<\epsilon$ (proof follows from the fact that if $U$ is big enough then it cannot shrink significantly as $f$ is topological mixing). Consequently, as $f_n$ converges uniformly to $f$, there exists $n_1\in\mathbb{N}$ such that if $U$ is open set such that $D_H(U,X)<\epsilon$ then $D_H(f_n(U),X)<\epsilon$ for all $n\geq n_1$. As $f$ is topological mixing there exists $k_0\in\mathbb{N}$ such that $D_H(f^n(U),X)<\epsilon~~\forall n\geq k_0$. Further, $\omega^n_{n+k_0}$ converges to $f^{k_0}$ and hence there exists $n_2\in\mathbb{N}$ such that $D(\omega^n_{n+k_0},f^{k_0})<\epsilon~~\forall n\geq n_2$. Thus if $r=\max\{n_1,n_2\}$ then for any $k\geq k_0,n\geq r$, triangle inequality ensures that $D_H(\omega^n_{n+k}(U), f^k(U))<2\epsilon$. As the argument holds for a sequence of smaller $\epsilon_n>0$ ($\epsilon_n\rightarrow0$), the discussions ensure collective convergence of $\omega^n_{n+k}(U)$ to $X$ and hence the non-autonomous system is topologically mixing. It may be noted that if non-autonomous system is topologically mixing then a similar set of arguments once again establishes collective convergence of $\omega^n_{n+k}(U)$ (to $X$) and hence topological mixing is equivalent for the two systems unconditionally.
\end{Remark}

\begin{Proposition}\label{tm}
Let $(X,\mathbb{F})$ be a non-autonomous system generated by a family $\mathbb{F}$ of feeble open maps. If $(f_n)$ converges uniformly to $f$ then, $(X,f)$ is topologically mixing $\Leftrightarrow$ $(X,\mathbb{F})$ is topologically mixing.
\end{Proposition}

\begin{proof}
The proof follows from discussions in Remark \ref{r}.
\end{proof}

\begin{Remark}
The above result establishes equivalence of topological mixing for the two systems if the family $\mathbb{F}$ is feeble open. The proof uses the fact that for a topologically mixing system, if an open set is "large" then it stays "large" for all times. Further as $f^r(U)$ and $f^s(U)$ are close for large $r,s$ and $\omega^n_{n+r}$ converges to $f^r$, $\omega^n_{n+k}$ converges collectively and hence the non-autonomous system is topologically mixing. It may be noted that any transitive map $f$ on $I=[a,b]$ is either topologically mixing or there exists a fixed point $c\in I$ such that for $P_1=[a,c],P_2=[c,b]$, $f(P_1)=P_2$, $f(P_2)=P_1$ and $f^2|_{P_i}$ is topologically mixing. As topologically mixing is equivalent for the two systems for feeble open $\mathbb{F}$, transitivity (and hence all forms of mixing) are equivalent for the two systems when $\mathbb{F}$ is feeble open. Hence we get the following corollary.
\end{Remark}

\begin{Cor}\label{tm1}
Let $I=[0,1]$ be the unit interval and let $(I,\mathbb{F})$ be a non-autonomous system generated by a feeble open family $\mathbb{F}$. If $(f_n)$ converges uniformly to $f$ then, $(I,f)$ is exhibits any form of mixing $\Leftrightarrow$ $(I,\mathbb{F})$ is exhibits identical form of mixing.
\end{Cor}

We now turn our attention to various forms of sensitivities for the two systems.

\begin{Proposition}\label{sen}
Let $(X,\mathbb{F})$ be a non-autonomous system generated by a family $\mathbb{F}$ of feeble open maps commuting with $f$. If $\sum \limits_{n=1}^{\infty} D(f_n,f) < \infty$ then $(X,f)$ is sensitive $\Leftrightarrow$ $(X,\mathbb{F})$ is sensitive.
\end{Proposition}

\begin{proof}
Let $\epsilon>0$ be given and let  $U=S(u,\epsilon)$ a non-empty open set in $X$. As $\sum \limits_{n=1}^{\infty} D(f_n,f) < \infty$, by corollay \ref{apcor} we obtain, for any $\epsilon>0$, there exists $n\in\mathbb{N}$ such that $d(\omega_{n+k}(x),f^{k}(\omega_n(x)))<\epsilon$ for all $x\in X, k\in\mathbb{N}$. \\

Let $\delta>0$ be constant of sensitivity for $f$ and let $m\in\mathbb{N}$ such that $\frac{1}{m}<\frac{\delta}{4}$. Thus, there exists $n_0\in\mathbb{N}$ such that $d(\omega_{n_0+k}(x),f^{k}(\omega_{n_0}(x)))<\frac{1}{m}$ for all $x\in X, k\in\mathbb{N}$. As $f_n$'s are feeble open, $\omega_{n_0}(U)$ is feeble open and hence sensitivity of $f$ implies there exists $v_1,v_2\in\omega_{n_0}(U)$ and $k\in\mathbb{N}$ such that $d(f^k(v_1),f^k(v_2))>\delta$. As $v_1,v_2\in\omega_{n_0}(U)$, there exists $v'_1,v_2'\in U$ such that $v_1=\omega_{n_0}(v_1')$ and $v_2=\omega_{n_0}(v_2')$ and hence $d(f^k(\omega_{n_0}(v_1')),f^k(\omega_{n_0}(v_2')))>\delta$. Also $d(\omega_{n_0+k}(v_i'), f^k(\omega_{n_0}(v_i')))<\frac{1}{m}$ for $i=1,2$ and thus by triangle inequality $d(\omega_{n_0+k}(v_1'),\omega_{n_0+k}(v_2'))>\delta-\frac{2}{m}>\frac{\delta}{2}$. Thus we obtain $v_1',v_2'\in U$ such that $d(\omega_{n_0+k}(v_1'),\omega_{n_0+k}(v_2'))>\frac{\delta}{2}$ and hence $(X,\mathbb{F})$ is sensitive.\\

Conversely, let $(X,\mathbb{F})$ be sensitive with sensitivity constant $\delta$ and let $U$ be non-empty open in $X$. Let $m\in\mathbb{N}$ such that $\frac{1}{m}<\frac{\delta}{4}$. Thus, there exists $n_0\in\mathbb{N}$ such that $d(\omega_{n_0+k}(x),f^{k}(\omega_{n_0}(x)))<\frac{1}{m}$ for all $x\in X, k\in\mathbb{N}$. For any $n\in\mathbb{N}$, as sensitivity of $(X,\mathbb{F})$ ensures existence of $k_n\in\mathbb{N}$ with $diam(\omega_{k_n}(S(x,\frac{1}{n})))>\delta$, the set $\{k:daim(\omega_k(S(x,\frac{1}{n})))>\delta\}$ is infinite. Consequently, for any non-empty open set $U$, the set of times $k$ when $diam(\omega_k(U))>\delta$ is infinite. Thus, for the open set $\omega_{n_0}^{-1}(U)$ , there exists $v_1,v_2\in\omega_{n_0}^{-1}(U)$ and $k\in\mathbb{N}$ such that $d(\omega_{n_0+k}(v_1),\omega_{n_0+k}(v_2))>\delta$. As $v_1,v_2\in\omega_{n_0}^{-1}(U)$, there exists $v'_1,v_2'\in U$ such that $v_1'=\omega_{n_0}(v_1)$ and $v_2'=\omega_{n_0}(v_2)$ and hence  $d(f_{n_0+k}\circ\ldots f_{n_0+1}(v_1'),f_{n_0+k}\circ\ldots f_{n_0+1}(v_2'))>\delta$. Also $d(\omega_{n_0+k}(v_i), f^k(\omega_{n_0}(v_i)))<\frac{1}{m}$ or $d(\omega_{n_0+k}(v_i), f^k(v_i'))<\frac{1}{m}$ for $i=1,2$ and thus by triangle inequality $d(f^k(v_1'),f^k(v_2'))>\delta-\frac{2}{m}>\frac{\delta}{2}$. Thus we obtain $v_1',v_2'\in U$ such that $d(f^k(v_1'),f^k(v_2'))>\frac{\delta}{2}$ and hence $(X,f)$ is sensitive.
\end{proof}

\begin{Remark}
The above result provides sufficient conditions under which the sensitivity of the two systems $(X,\mathbb{F})$ and $(X,f)$ is equivalent. Once again, the proof uses the collective convergence of the set of sequences $\{(\omega^n_{n+k}):k\in\mathbb{N}\}$ and neither commutativity nor the rate of convergence is used explicitly to establish the results. It is established that if the collective convergence is guaranteed, feeble openness ensures that sensitivity of $(X,f)$ implies sensitivity of $(X,\mathbb{F})$. However for the converse part, feeble openness is redundant and collective convergence is sufficient to preserve the sensitivity in the other direction. Further, the choice of $m (\frac{1}{m}<\frac{\delta}{4})$ is arbitrary and can be made finer which in turn ensures greater separability of points for the systems considered. Hence the constant of sensitivity is preserved and the two systems are sensitive with same constant of sensitivity. Further, the above proof establishes that for any pair of distinct points $x,y\in X$ there exists points $x',y'$ such that the set of times of separation of $x$ and $y$ in the non-autonomous system contains the translates (by a fixed number depending on the rate of convergence) of set of times of separation of $x'$ and $y'$ in the autonomous system (and vice-versa). In particular, a similar argument establishes that for any non-empty open set $U$, there exists a non-empty open set $U'$ such that the set of times of expansivity of $U$ in the non-autonomous system contains the translates (by a fixed number depending on the rate of convergence) of set of times of expansitivity of $U'$ in the autonomous system and vice-versa. Consequently, if the autonomous(non-autonomous) system is cofinitely sensitive then non-autonomos(autonomous) system is also cofinite sensitive and hence cofinite sensitivity is equivalent for the two systems.
\end{Remark}

\begin{Cor}\label{ss}
Let $(X,\mathbb{F})$ be a non-autonomous system generated by a family $\mathbb{F}$ of feeble open maps commuting with $f$. If $\sum \limits_{n=1}^{\infty} D(f_n,f) < \infty$ then $(X,f)$ is cofinitely sensitive $\Leftrightarrow$ $(X,\mathbb{F})$ is cofinitely sensitive.
\end{Cor}

\begin{Remark}
The above proof establishes the equivalence of sensitivity (and strong sensitivity) for the two systems $(X,\mathbb{F})$ and $(X,f)$. However, the proof once again utilizes the collective convergence guaranteed by "commutative condition" and "fast convergence" and does not use any of the imposed conditions explicitly. Hence the above conclusions hold good under collective convergence and can hold good without "commutative condition" and "fast convergence". Further, it is worth noting that feeble openness is a necessary condition for the results to hold good and hence cannot be dropped.  We now give some examples in support of our statement.
\end{Remark}

\begin{ex}
Let $S^1$ be the unit circle and let $f_n:S^1\rightarrow S^1$ be defined as $f_n(\theta) = 2\theta+\frac{1}{n}$. Then $(f_n)$ converges uniformly to the angle doubling map (say $f$). It may be noted that $\sum \limits_{n=1}^{\infty} D(f_n,f)$ is infinite and the family $\mathbb{F}$ does not commute with the limit map. However, if $\mathbb{F}=\{f_n: n\in\mathbb{N}\}$ then both $(X,\mathbb{F})$ and $(X,f)$ exhibit all forms of mixing. Further, as both $(X,\mathbb{F})$ and $(X,f)$ exhibit sensitivity (and cofinite sensitivity), neither of the conditions imposed is necessary for preserving any kind of mixing or sensitivity.
\end{ex}

\begin{ex}\label{sens}
Let $I$ be the unit interval and let $f:I\rightarrow I$ be defined as

$f(x) = \left\{%
\begin{array}{ll}
            2x  & \text{for x} \in [0, \frac{1}{2}] \\
           2-2x & \text{for x} \in [\frac{1}{2}, 1] \\
\end{array} \right.$ and let

$g(x) = \left\{%
\begin{array}{ll}
            1  & \text{for x} \in [0, \frac{1}{2}] \\
           2-2x & \text{for x} \in [\frac{1}{2}, 1]
\end{array} \right.$ \\

Let $(X,\mathbb{F})$ be the non-autonomous system generated by $\mathbb{F}=\{g,f,f,\ldots\}$. It may be noted that as $\omega^n_{n+k}=f^k$, collective convergence of $\{(\omega^n_{n+k}):k\in\mathbb{N}\}$ is ensured.  However, as $(X,f)$ exhibits all forms of mixing and sensitivities but $(X,\mathbb{F})$ does not exhibit any form of mixing or sensitivity, feeble openness of the family $\mathbb{F}$ is a necessary condition for preserving any notion of mixing or sensitivity.
\end{ex}

\begin{Proposition}\label{pp}
Let $(X,\mathbb{F})$ be a non-autonomous system generated by a family $\mathbb{F}$. If $(f_n)$ converges uniformly to $f$, then, $x$ is periodic for $(X,\mathbb{F}) \Rightarrow x$ is periodic for $(X,f)$.
\end{Proposition}

\begin{proof}
Let $x_0$ be periodic for $(X,\mathbb{F})$ with period $k$ and let $\epsilon>0$ be given. As $\omega^n_{n+k}$ converges uniformly to $f^k$, there exists $n_0\in\mathbb{N}$ such that $D(\omega^n_{n+k},f^k)<\epsilon~~\forall n\geq n_0$. Therefore, $d(\omega^{n_0k}_{n_0k+k}(x),f^k(x))<\epsilon$ for any $x\in X$ and hence $d(\omega^{n_0k}_{n_0k+k}(\omega_{n_0k}(x_0)),f^k(\omega_{n_0k}(x_0)))<\epsilon$. As $\omega_{rk}(x_0)=x_0$ for all $r$, the above argument ensures $d(f^k(x_0),x_0)<\epsilon$. As the result holds good for any $\epsilon>0$ we have $f^k(x_0)=x_0$ and hence $x_0$ is periodic for $f$ with period $k$.
\end{proof}

%

\begin{Remark}
The result establishes that if $x$ is periodic for $(X,\mathbb{F})$ then $x$ is periodic for $(X,f)$. The proof utilizes the fact that if $(f_n)$ converges uniformly to $f$ then for any fixed $k\in\mathbb{N}$, $\omega^n_{n+k}$ converges to $f^k$. As the result utilizes the convergence of $\omega^n_{n+k}$ for a fixed $k$, the proof requires only the uniform convergence of $(f_n)$ and is true without any other additional assumptions. Consequently if $(X,\mathbb{F})$ has dense set of periodic points then $(X,f)$ also exhibits dense set of periodic points. However, the results only establishes the preservance of periodic points in one direction and existence of periodic points (or dense set of periodic points) is not equivalent for the two systems.
\end{Remark}

\begin{Cor}\label{pp2}
Let $(X,\mathbb{F})$ be a non-autonomous system generated by a family $\mathbb{F}$. If $(f_n)$ converges uniformly to $f$, then, $(X,\mathbb{F})$ has dense set of periodic points $\Rightarrow(X,f)$ has dense set of periodic points.
\end{Cor}

\begin{ex}
Let $S^1$ denote the unit circle and let $\mathbb{F}=\{f_n:n\in\mathbb{N}\}$ where $f_n:S^1\rightarrow S^1$ is defined as $f_n(\theta)=\theta+\frac{1}{n^2}$. Then $f_n$'s are rotations on unit circle converging uniformly to identity $I$ and hence every point is periodic for $(S^1,I)$. However, as $\sum\limits_{n=1}^{\infty}\frac{1}{n^2}= \frac{\pi^2}{6}<2\pi$, the non-autonomous system $(S^1,\mathbb{F})$ does not have any periodic point and hence the converse of the above result is not true.
\end{ex}

\begin{Proposition}\label{nads}
Let $X$ be compact and let $(X,\mathbb{F})$ be a sensitive non-autonomous system generated by a family $\mathbb{F}$. Then for any $x\in X$, $Prox_{\mathbb{F}}(x$) is dense in $X\Leftrightarrow$ $LY_{\mathbb{F}}(x$) is dense in $X$.
\end{Proposition}

\begin{proof}
For any $x\in X$, it may be noted that if $(X,\mathbb{F})$ is sensitive with sensitivity constant $\delta$, then for any $y\in X$ and any neighborhood $U_y$ of $y$, there exists $y'\in U_y$ and $k\in\mathbb{N}$ such that $d(\omega_n(y),\omega_n(y'))>\delta$. By triangle inequality, $d(\omega_n(x),\omega_n(y))>\frac{\delta}{2}$ or $d(\omega_n(x),\omega_n(y'))>\frac{\delta}{2}$ and hence the set of points $\frac{\delta}{2}$-sensitive to $x$ are dense in $X$.

Let $\epsilon>0$ be fixed. For any $x\in X$ and any non-empty open subset $U$ of $X$, let $V$ be non-empty open such that $V\subset\overline{V}\subset U$. As $Prox_{\mathbb{F}}(x)$ is dense in $X$, there exists $y\in V$ such that the pair $(x,y)$ is proximal and hence there exists $n_1\in\mathbb{N}$ such that $d(\omega_{n_1}(x),\omega_{n_1}(y))<\epsilon$. By continuity, there exists a neighborhood $U_1(\subset V)$ of $y$ such that $d(\omega_{n_1}(x),\omega_{n_1}(u_1))<\epsilon$ for all $u_1\in U_1$. As the set of points $\frac{\delta}{2}$-sensitive to $x$ are dense in $X$, there exists $y_1\in U_1$ and $m_1\in\mathbb{N}$ such that $d(\omega_{m_1}(x),\omega_{m_1}(y_1))>\frac{\delta}{2}$ and once again by continuity there exists a neighborhood $V_1(\subset U_1\subset V)$ of $y_1$ such that $d(\omega_{n_1}(x),\omega_{n_1}(v_1))<\epsilon$  and $d(\omega_{m_1}(x),\omega_{m_1}(v_1))>\frac{\delta}{2}$ for all $v_1\in V_1$. Hence for $\epsilon>0$ and any pair $x,U$ (where $x\in X$ and $U$ is non-empty open subset of $X$) there exists $n,m\in\mathbb{N}$ and a non-empty open subset $U_{\epsilon}$ of $X$ (satisfying $U_{\epsilon}\subset \overline{U_{\epsilon}}\subset U$) such that $d(\omega_{m}(x),\omega_{m}(u))>\frac{\delta}{2}$ and $d(\omega_{n}(x),\omega_{n}(u))<\epsilon$ for all $u\in U_{\epsilon}$.

Let $x\in X$ and $U$ be any non-empty open subset of $X$. By argument above there exists non-empty open set $U_1$, $U_1\subset \overline{U_1}\subset U$ and $n_1,m_1\in\mathbb{N}$ such that $d(\omega_{m_1}(x),\omega_{m_1}(y))>\frac{\delta}{2}$ and $d(\omega_{n_1}(x),\omega_{n_1}(y))<\frac{1}{2}$ for all $y\in U_{1}$. Repeating the process for the pair $(x,U_1)$, there exists $U_2$ satisfying $U_2\subset \overline{U_2}\subset U_1$ and $n_2,m_2\in\mathbb{N}$ such that $d(\omega_{m_2}(x),\omega_{m_2}(y))>\frac{\delta}{2}$ and $d(\omega_{n_2}(x),\omega_{n_2}(y))<\frac{1}{4}$ for all $y\in U_{2}$. Inductively, we obtain a decreasing sequence $U_k$ of non-empty open subsets of $X$ such that $U_k\subset\overline{U_k}\subset U_{k-1}$ and sequences $(n_k),(m_k)\in\mathbb{N}$ such that $d(\omega_{m_k}(x),\omega_{m_k}(y))>\frac{\delta}{2}$ and $d(\omega_{n_k}(x),\omega_{n_k}(y))<\frac{1}{2^k}$ for all $y\in U_{k}$. As $X$ is compact $\bigcap\limits_{k=1}^{\infty} \overline{U_k}\neq \phi$. Then for any $u\in\bigcap\limits_{k=1}^{\infty}\overline{U_k}$, we have $d(\omega_{m_k}(x),\omega_{m_k}(u))>\frac{\delta}{2}$ and $d(\omega_{n_k}(x),\omega_{n_k}(u)<\frac{1}{2^k}$ for all $k$ and hence $(x,u)$ is a Li-Yorke pair. As the argument holds good for any non-empty open set $U$, $LY_{\mathbb{F}}(x$) is dense in $X$.

As every Li-Yorke pair is proximal, proof of the converse part is trivial.
\end{proof}
\begin{Remark}
The result establishes that for sensitive systems, if proximal cells are dense in $X$ for $(X,\mathbb{F})$ then Li-Yorke cells are also dense. It may be noted that the result does not use the compactness of the space $X$ completely and holds good for locally compact spaces also. Further, the proof does not use denseness of the proximal cell completely and establishes that for a sensitive system, if proximal cell of a point $x$ is dense in a neighborhood of $x$ then Li-Yorke cell is dense in a neighborhood of $x$. Consequently, a similar argument establishes that for a sensitive system, if proximal cell of a point $x$ is dense in a neighborhood of $x$ then $x$ is point of Li-Yorke sensitivity. The result is a natural extension of the result established in \cite{kol} for the autonomous systems.
\end{Remark}

\begin{Cor}
Let $(X,f)$ be a compact sensitive system. Then for any $x\in X$, $Prox(x)$ is dense in $X \Leftrightarrow LY(x)$ is dense in X.
\end{Cor}

\begin{Proposition}\label{pr}
Let $(X,\mathbb{F})$ be a non-autonomous system generated by a family $\mathbb{F}$ commuting with $f$. If $\sum \limits_{n=1}^{\infty} D(f_n,f) < \infty$ then, $(x,y)$ is proximal for $(X,f) \Rightarrow(x,y)$ is proximal for $(X,\mathbb{F})$.
\end{Proposition}

\begin{proof}
Let $\epsilon>0$ be given and let $(x,y)$ be proximal for $(X,f)$. As $(x,y)$ is proximal, there exists a sequence $(n_k)$ in $\mathbb{N}$ such that $\lim\limits_{n_k\rightarrow\infty} d(f^{n_k}(x),f^{n_k}(y))=0$. As $\sum \limits_{n=1}^{\infty} D(f_n,f) < \infty$, choose $r\in\mathbb{N}$ such that $\sum \limits_{n=r}^{\infty} D(f_n,f) < \frac{\epsilon}{3}$. As $\omega_r$ is continuous (and hence uniformly continuous), there exists $\eta>0$ such that for any $u,v\in X$, $d(u,v)<\eta$ implies $d(\omega_r(u),\omega_r(v))<\frac{\epsilon}{3}$. As $(x,y)$ is proximal for $(X,f)$, there exists $n_k\in\mathbb{N}$ such that $d(f^{n_k}(x),f^{n_k}(y))<\eta$ and hence $d(\omega_r(f^{n_k}(x)),\omega_r(f^{n_k}(y)))<\frac{\epsilon}{3}$ or $d(f^{n_k}(\omega_r(x)),f^{n_k}(\omega_r(y)))<\frac{\epsilon}{3}$ (by commutativity). Also by corollary \ref{apcor}, $d(\omega_{r+n_k}(u),f^{n_k}(\omega_r(u)))< \sum \limits_{i=1}^{n_k} D(f_{r+i},f)<\frac{\epsilon}{3}$ for any element $u$ in $X$. Thus by triangle inequality, $d(\omega_{r+n_k}(x),\omega_{r+n_k}(y))\leq d(\omega_{r+n_k}(x),f^{n_k}(\omega_r(x)))+ d(f^{n_k}(\omega_r(x)),f^{n_k}(\omega_r(y)))+d(f^{n_k}(\omega_r(y)),f^{n_k}(\omega_r(y)))<\epsilon$. As the proof works for any $\epsilon>0$, $(x,y)$ is proximal for $(X,\mathbb{F})$.
\end{proof}

\begin{Proposition}\label{prox}
Let $X$ be compact and let  $(X,\mathbb{F})$ be a non-autonomous system generated by a family $\mathbb{F}$ commuting with $f$. If $\sum \limits_{n=1}^{\infty} D(f_n,f) < \infty$ then, proximal cell of each $x$ is dense for $(X,f)$ $\Leftrightarrow$ proximal cell of each $x$ is dense for $(X,\mathbb{F})$.
\end{Proposition}

\begin{proof}
As every pair proximal for $(X,f)$ is proximal for $(X,\mathbb{F})$, if Prox($x$) is dense for $(X,f)$ then Prox($x$) is also  dense for $(X,\mathbb{F})$ and the proof of forward part is complete.

Conversely let $Prox_{\mathbb{F}}(x)$ be dense for each $x \in X$. Fix $x\in X$ and let $U$ be a non-empty open subset of $X$. Let $U_1$ be non-empty open such that $\overline{U_1}\subset U$. 
As $\sum \limits_{n=1}^{\infty} D(f_n,f) < \infty$, for each $m\in \mathbb{Z}^+$ there exists $r_m\in\mathbb{N}$ such that $\sum \limits_{n=r_m}^{\infty} D(f_n,f) < \frac{1}{3.2^m}$. Fix $m=1$ and let $V_{m}=\omega_{r_m}^{-1}(U_{m})$. Then $V_{m}$ is open in $X$. Pick any $x_m\in\omega_{r_m}^{-1}(x)$. As $Prox_{\mathbb{F}}(x_m)$ is dense for $(X,\mathbb{F})$, there exists a $z_{m}=\omega_{r_m}^{-1}(y_{m})\in V_{m}~~(y_{m}\in U_m)$ such that $(x_m,z_{m})$ is proximal for $(X,\mathbb{F})$. Thus, there exists a sequence $(n_{k,m})$ in $\mathbb{N}$ such that $\lim\limits_{n_{k,m}\rightarrow\infty}d(\omega_{n_{k,m}+r_{m}}(x_m),\omega_{n_{k,m}+r_{m}}(z_{m}))=0$. Choose $s_{m}\in\mathbb{N}$ such that $d(\omega_{s_{m}+r_{m}}(x_m),\omega_{s_{m}+r_{m}}(z_{m}))<\frac{1}{3.2^{m}}$. Also, by corollary \ref{apcor}, for any $w\in X$, $d(\omega_{s_{m}+r_{m}}(w), f^{s_{m}}(\omega_{r_{m}}(w)))< \sum \limits_{i=1}^{s_{m}} D(f_{r_{m}+i},f)<\frac{1}{3.2^{m}}$ and hence by triangle inequality, we have $d(f^{s_{m}}(\omega_{r_m}(x_m)),f^{s_{m}}(\omega_{r_m}(z_{m})))< \frac{1}{2^m}$ or $d(f^{s_{m}}(x),f^{s_{m}}(y_{m}))< \frac{1}{2^{m}}$. By continuity, there exists neighborhood $U_{m+1}$ of $y_{m}$ such that $\overline{U_{m+1}}\subset U_{m}$ and $d(f^{s_{m}}(\omega_{r_m}(x)),f^{s_{m}}(\omega_{r_m}(y)))< \frac{1}{2^m}$ for all  $y\in U_{m+1}$.

Repeating the process for pair $(x,U_m)$ (for each $m$), we obtain a open set $U_{m+1}$, $\overline{U_{m+1}}\subset U_{m}\subset U$ and $s_m\in\mathbb{N}$ such that $d(f^{s_m}(x),f^{s_{m}}(y))< \frac{1}{2^{m}}$ for all $y\in U_{m+1}$.  As $\overline{U_{m+1}}$ is a nested decreasing sequence of closed sets in a compact metric space and hence $\bigcap \limits_{m=1}^{\infty} \overline{U_{m}}\neq\phi$. Then, for any $u\in \bigcap \limits_{m=1}^{\infty} \overline{U_{m}}\subset U$, as $u\in U_{m}$ for all $m$, we have $d(f^{s_m}(x),f^{s_m}(u))< \frac{1}{2^m}$ for all $m$ and hence the pair $(x,u)$ is proximal for $(X,f)$. As the proof holds good for any non-empty open set $U$ in $X$, $Prox_f(x)$ is dense in $X$.
\end{proof}

\begin{Remark}\label{rem}
The above proof establishes that proximal cell of each $x$ is dense for $(X,f)$ if and only if proximal cell of each $x$ is dense for $(X,\mathbb{F})$ and hence the dynamical behavior of two systems in equivalent in this regard. However, the proof does not exploit the denseness condition and establishes that denseness of proximal cells in some neighborhood is equivalent for the two systems. Further, by Proposition \ref{nads}, if a system is sensitive then denseness of proximal cells and Li-Yorke cells is equivalent for the two systems. Equivalently, if proximal cells are dense for a system then the system is sensitive if and only if it is Li-Yorke sensitive. As sensitivity is equivalent for feeble open systems (with sufficiently fast rate of convergence), we obtain the following corollary.
\end{Remark}

\begin{Cor}\label{lys}
Let $(X,\mathbb{F})$ be a non-autonomous system generated by a family $\mathbb{F}$ of feeble open maps commuting with $f$ satisfying $\sum \limits_{n=1}^{\infty} D(f_n,f) < \infty$. If $Prox(x)$ is dense in $X$ for each $x\in X$, then $(X,f)$ is Li-Yorke sensitive $\Leftrightarrow$ $(X,\mathbb{F})$ is Li-Yorke sensitive.
\end{Cor}

\begin{proof}
The result is a direct consequence of Proposition \ref{sen} and Remark \ref{rem}.
\end{proof}

We now establish equivalence of denseness of proximal pairs for the two systems.

\begin{Proposition}\label{prp}
Let $X$ be compact and let $(X,\mathbb{F})$ be a non-autonomous system generated by a family $\mathbb{F}$ commuting with $f$. If $\sum \limits_{n=1}^{\infty} D(f_n,f) < \infty$ then, set of proximal pairs is dense for $(X,f) \Leftrightarrow$ set of proximal pairs is dense for $(X,\mathbb{F})$.
\end{Proposition}

\begin{proof}
As every pair proximal for $(X,f)$ is proximal for $(X,\mathbb{F})$, if $(X,f)$ has dense set of proximal pairs then $(X,\mathbb{F})$ has dense set of proximal pairs.

Conversely, let $(X,\mathbb{F})$ exhibit dense set of proximal pairs and let $U_1\times U_2$ be any non-empty open set in $X\times X$. Let $U_i'$ be non-empty open such that $\overline{U_i'}\subset U_i$. As $\sum \limits_{n=1}^{\infty} D(f_n,f) < \infty$, for each $m\in \mathbb{Z}^+$ there exists $r_m\in\mathbb{N}$ such that $\sum \limits_{n=r_m}^{\infty} D(f_n,f) < \frac{1}{3.2^m}$.Fix $m=0$ and let $U_{1,m}=U_1'$ and $U_{2,m}=U_2'$.  For $i=1,2$, let $V_{i,m}=\omega_{r_m}^{-1}(U_{i,m})$. Then $V_{1,m}\times V_{2,m}$ is open in $X\times X$. Consequently there exists a $(x_{1,m},x_{2,m})\in V_{1,m}\times V_{2,m}$ such that $(x_{1,m},x_{2,m})$ is proximal for $(X,\mathbb{F})$ and hence there exists a sequence $(n_{k,m})$ in $\mathbb{N}$ such that $\lim\limits_{n_{k,m}\rightarrow\infty}d(\omega_{n_{k,m}+r_m}(x_{1,m}),\omega_{n_{k,m}+r_m}(x_{2,m}))=0$. Choose $s_m\in\mathbb{N}$ such that $d(\omega_{s_m+r_m}(x_{1,m}),\omega_{s_m+r_m}(x_{2,m}))<\frac{1}{3.2^m}$. Also, by corollary $2$, $d(\omega_{s_m+r_m}(x_{i,m}), f^{s_m}(\omega_{r_m}(x_{i,m})))< \sum \limits_{i=1}^{s_m} D(f_{r_m+i},f)<\frac{1}{3.2^m}$ and hence by triangle inequality, we have $d(f^{s_m}(\omega_{r_m}(x_{1,m})),f^{s_m}(\omega_{r_m}(x_{2,m})))< \frac{1}{2^m}$.

Note that $\omega_{r_m}(x_{i,m})\in U_{i,m}$ and assuming $u_{i,m}=\omega_{r_m}(x_{i,m})$ yields $d(f^{s_m}(u_{1,m}),f^{s_m}(u_{2,m}))< \frac{1}{2^m}$.  Thus, there exists neighborhoods $U_{i,m+1}$ of $u_{i,m}$ such that $\overline{U_{i,m+1}}\subset U_{i,m}$ and $d(f^{s_m}(\omega_{r_m}(x)),f^{s_m}(\omega_{r_m}(y)))< \frac{1}{2^m}$ for all $x\in U_{1,m+1}, y\in U_{2,m+1}$.

Repeating the process for each $m$, we obtain a nested sequence of open sets $U_{i,m+1}$, $\overline{U_{i,m+1}}\subset U_{i,m}\subset U_i$ such that $d(f^{s_m}(\omega_{r_m}(x)),f^{s_m}(\omega_{r_m}(y)))< \frac{1}{2^m}$ for all $x\in U_{1,m+1}, y\in U_{2,m+1}$.  As $\overline{U_{i,m+1}}$ is a nested decreasing sequence of closed sets in a compact metric space and hence $\bigcap \limits_{m=0}^{\infty} \overline{U_{i,m}}\neq\phi$. Then, for any $u_1\in \bigcap \limits_{m=0}^{\infty} \overline{U_{1,m}}\subset U_1$, $u_2\in \bigcap \limits_{m=0}^{\infty} \overline{U_{2,m}}\subset U_2$, as $u_i\in U_{i,m}$ for all $m$, we have $d(f^{s_m}(u_{1}),f^{s_m}(u_{2}))< \frac{1}{2^m}$ for all $m$ and hence $(u_1,u_2)$ is proximal for $(X,f)$. As the proof holds good for any pair of non-empty open sets $U_1,U_2$ in $X$, the set of pairs proximal for $(X,f)$ is dense in $X\times X$.
\end{proof}

\begin{Remark}\label{rem1}
The above proofs establish the equivalence of denseness of proximal cells (pairs) for the two systems $(X,\mathbb{F})$ and $(X,f)$. The proofs establish that under stated conditions, proximal pairs (each proximal cell) are dense in $(X,\mathbb{F})$ if and only if proximal pairs (each proximal cell) are dense in $(X,f)$. It may be noted that to establish equivalence of proximal cells (pairs) "commutative condition" is explicitly used and hence is needed for establishing the stated results. However, if the family $\mathbb{F}$ is feeble open, a similar argument establishes identical results under collective convergence of $\{(\omega^n_{n+k}):k\in\mathbb{N}\}$. Hence identical results can be obtained for a feeble open family $\mathbb{F}$ if collective convergence is guaranteed and the"commutative condition" is not needed to establish the stated results. The discussions also establish that if the family $\mathbb{F}$ is feeble open and each proximal cell is dense then Li-Yorke sensitivity is equivalent for the two systems. Consequently, the results established hold good when the family $\mathbb{F}$ is feeble open and the collection $\{(\omega^n_{n+k}):k\in\mathbb{N}\}$ converges collectively. Further, feeble openness is once again a necessary condition to establish equivalence of Li-Yorke sensitivity for the two systems and the result does not hold good if feeble openness of family $\mathbb{F}$ is dropped (as noted in Example \ref{sens}). The discussions lead to the following corollary.
\end{Remark}

\begin{Cor}
Let $X$ be compact and let $(X,\mathbb{F})$ be a non-autonomous system generated by a feeble open family $\mathbb{F}$. If $\{(\omega^n_{n+k}):k\in\mathbb{N}\}$ converges collectively then, each proximal cell (proximal pairs) is dense for $(X,f) \Leftrightarrow$ each proximal cell (proximal pairs) is dense for $(X,\mathbb{F})$.
\end{Cor}

\begin{Remark}
The above results establish that the dynamical behavior of $(X,\mathbb{F})$ and $(X,f)$ are closely related when the family of sequences $\{(\omega^n_{n+k}):k\in\mathbb{N}\}$ is collectively convergent. Such a condition is guaranteed when $f_n$ converges at a "sufficiently fast rate" and the family $\mathbb{F}$ commutes with the limit map $f$ and hence the results derived hold good when these two conditions are satisfied. As observed, such a condition is also guaranteed (under fast convergence) when the limit map $f$ is an isometry. However, "sufficiently fast rate of convergence" is not necessarily needed to establish collective convergence (Example \ref{eqex}) and collective convergence is guaranteed if the cumulative error can be made arbitrarily small. It may be noted that the conditions implying collective convergence are natural and can arise naturally in many applications. For example, if the non-autonomous system is generated by a commutative family $\mathbb{F}$, commutativity with the limit function is guaranteed. Further, if $f_n$ converges uniformly to $f$ and $f^n=I$ (the identity function) for some $n\in\mathbb{N}$, the validity of the derived results is guaranteed (under fast convergence). More generally, if $f_n$ converge uniformly to an isometry $f$ then, convergence at a sufficiently fast rate guarantees identical dynamical behavior of the two systems under consideration. The results can be applied to theory of group actions when the limit function is in the center of the acting group and interesting dynamics can be observed in this case. More generally, the results are true for a more generalized set of assumptions which are in general difficult to verify. In any case, the results derived can be visualized in many interesting situations and can be applied to many natural phenomena.
\end{Remark}

\section{Conclusion}
In this paper, dynamics of a non-autonomous system generated by a uniformly convergent sequence is investigated. In the process, we relate the dynamics of the non-autonomous system $(X,\mathbb{F})$ with the dynamics of the limiting system $(X,f)$. It is observed that if the family $\mathbb{F}$ commutes with $f$ and $f_n$ converges to $f$ at a "sufficiently fast rate" then many of the dynamical properties coincide for the two systems $(X,\mathbb{F})$ and $(X,f)$. In the process, equivalence of properties like equicontinuity, minimality, transitivity, weakly mixing, various notions of sensitivities, denseness of proximal cells and denseness of proximal pairs is established. While we establish that topological mixing for the two systems is equivalent for feeble open $\mathbb{F}$, we prove that if $(X,\mathbb{F})$ has a dense set of periodic points then $(X,f)$ has a dense set of periodic points. We give an example to show that the result does not hold in the opposite direction. We prove that if the limit map $f$ is an isometry (shrinking), the established results hold good without the "commutative condition". More generally, we prove that the equivalence of the dynamics of $(X,f)$ and $(X,\mathbb{F})$ can be established under the weaker assumption of collective convergence. We establish that denseness of proximal cell is equivalent to denseness of Li-Yorke cells for sensitive systems. In the process, we generalize the result obtained in \cite{kol} to the non-autonomous case. We prove that denseness of proximal cells (proximal pairs) is equivalent for the two systems under collective convergence. We also give examples to investigate the necessity of the conditions imposed.

\section*{Acknowledgement}
The authors thank the referee for his/her useful suggestions and remarks. The first author thanks National Board for Higher Mathematics (NBHM) for financial support.

\bibliography{xbib}

\end{document}